\documentclass[11pt,a4paper]{amsart}
\usepackage{amsmath,amsfonts,amssymb,amsthm}
\usepackage{mathrsfs}
\usepackage[all,cmtip]{xy}
\usepackage{xcolor}

\usepackage[left=0.8in,right=0.8in,top=0.8in,bottom=0.8in]{geometry}

\newtheorem{theorem}{Theorem}[section]
\newtheorem{lemma}[theorem]{Lemma}
\newtheorem{corollary}[theorem]{Corollary}
\newtheorem{proposition}[theorem]{Proposition}

\theoremstyle{definition}
\newtheorem{example}[theorem]{Example}
\newtheorem{definition}[theorem]{Definition}
\newtheorem{remark}[theorem]{Remark}

\newcommand{\Z}{{\mathbb Z}}
\newcommand{\N}{{\mathbb N}}
\newcommand{\R}{{\mathbb R}}

\newcommand{\C}{{\mathbb C}}

\def\kal#1{\mathcal{K}_{#1}}   

\def\hull#1{\langle{#1}\rangle} 
\def\rhull#1{\langle{#1}\rangle_\R}

\def\norm#1{\left|\left|{#1}\right|\right|}

\def\uar#1#2{\mathcal{U}_{\arg}({#1},{#2})}
\def\uab#1#2{\mathcal{U}_{\mathrm{abs}}({#1},{#2})}
\def\u#1#2{\mathcal{U}\left({#1},{#2}\right)}

\renewcommand{\S}{\mathbb{S}}

\title{Topologically independent sets in topological groups and vector spaces}

\author[J. Sp\v{e}v\'ak]{Jan Sp\v{e}v\'ak}
\address[Jan Sp\v{e}v\'ak]{Department of Applied Mathematics\\ Faculty of Information Technology\\ Czech Technical University in Prague\\ Th\'akurova 9, 16000, Prague 6\\ Czech Republic} \email{spevajan@fit.cvut.cz}

\thanks{The work on this paper was supported by the Czech Science Foundation (GA\v{C}R) grant 22-32829S}

\begin{document}

\begin{abstract}
We study topological versions of an independent set in an abelian group and a linearly independent set in a vector space, a {\em topologically independent set} in a topological group and a {\em topologically
linearly independent set} in a topological vector space. These counterparts of their algebraic versions
are defined analogously and possess similar properties.

Let $\C^\times$ be the multiplicative group of the field of complex numbers with its usual topology. 
We prove that a subset $A$ of an arbitrary Tychonoff power of $\C^\times$ is topologically independent if and only if the topological subgroup $\hull{A}$ that it generates is the Tychonoff direct sum $\bigoplus_{a\in A}\hull{a}$.

This theorem substantially generalizes an earlier result of the author, who has proved this for Abelian precompact groups.

Further, we show that topologically independent and topologically linearly independent sets coincide
in vector spaces with weak topologies, although they are different in general.

We characterize topologically linearly independent sets in vector spaces with weak topologies and normed spaces. In a weak topology, a set $A$ 
 is topologically linearly independent if
and only if its linear span is the Tychonoff direct sum  $\R^{(A)}$. In normed spaces $A$ is topologically linearly independent if and only if it is uniformly minimal. Thus, from the point of view of topological
linear independence, the Tychonoff direct sums $\R^{(A)}$ and (linear spans of) uniformly minimal sets, which are closely related to bounded biorthogonal systems, are of the same essence.
\end{abstract}



\maketitle

\def\rank#1{trank(#1)}

Throughout this paper, we consider four categories of objects and their morphisms: Abelian groups and homomorphisms, vector spaces over the field $\R$ of real numbers and linear maps, (Abelian) Hausdorff topological groups and continuous homomorphisms, and (Hausdorff) topological vector spaces with continuous linear maps. 

Given a subset $A$ of a group $G$, we use the symbol $\hull{A}$ to denote the subgroup of $G$ generated by $A$. If $G$ is also a vector space, then $\rhull{A}$ stays for the vector subspace generated by $A$. For simplicity, we write $\hull{a}$ and $\rhull{a}$ rather than $\hull{\{a\}}$ and $\rhull{\{a\}}$, respectively, for all $a\in G$. The neutral (zero) element of a group or a vector space $G$ is denoted by $e_G$ or $0_G$.

Given a nonempty set $I$ and a family $\{G_i\colon i\in I\}$ of abelian groups, the symbol $\bigoplus_{i\in I}G_i$ stays for the direct sum of this family. It is a group with the group operations defined coordinate-wise. Suppose $G_i$ is a topological group for each $i\in I$. In that case, we endow this direct sum with the subspace topology inherited from the product $\prod_{i\in I}G_i$ endowed with the Tychonoff product topology. We call $\bigoplus_{i\in I}G_i$ a Tychonoff direct sum, a topological group. Similarly, if $G_i$ is a topological vector space for every $i\in I$, the latter Tychonoff direct sum becomes a topological vector space. 
If $G_i=G$ for all $i\in I$, we denote the (Tychonoff) direct sum by $G^{(I)}$.

Following \cite{DSS}, we define the Kalton map $\kal{A}$ associated with $A$ to be the unique group homomorphism $$\kal{A}\colon\bigoplus_{a\in A}\hull{a}\to G,$$
 which extends each natural inclusion map $\hull{a}\to G$ for $a\in A$.

The symbol $\C^\times$ denotes the multiplicative group of the field of complex numbers with its usual Euclidean topology, by $\S$  we denote its compact subgroup of all complex numbers $z$ with $|z|=1$, and $\R_+$ stays for the topological subgroup of $\C^\times$ consisting of positive real numbers. The symbol $\Z$ denotes the discrete group of integers, and $\N$ is the set of natural numbers.

\section{Introduction}

One of the fundamental concepts in the theory of abelian groups is that of an independent set. A subset $A$ of nonzero elements of a group $G$ is {\em independent}, provided that for every finite subset $B\subseteq A$ and every indexed set $\{z_b\colon b\in B\}$ of integers, the equality $\sum_{b\in B}z_bb=e_G$ implies that $z_bb=e_G$ for all $b\in B$.

When we replace the algebraic relation of ``being equal to zero'' with a topological notion of ``being close to zero,'' we obtain a natural group-topological counterpart of an independent set from \cite[Definition 4.1]{DSS}:

\begin{definition}
A subset $A$ of nonzero elements of a topological group $G$ is {\em topologically independent} if, for every neighborhood $W$ of $e_G$ there is a neighborhood $U$ of $e_G$ such that for every finite subset $B\subseteq A$ and every indexed set $\{z_b\colon b\in B\}$ of integers, the inclusion $\sum_{b\in B}z_bb\in U$ implies that $z_bb\in W$ for all $b\in B$. The neighborhood $U$ is called a {\em $W$-witness} of the topological independence of $A$.
\end{definition}

Similarly, in vector spaces, we can replace the integer linear combination in the definition of an independent set with a linear combination with real numbers as coefficients to obtain one of the key notions of linear algebra -- a linearly independent set. The definition of its topological counterpart was anticipated in \cite[Remark 4.13]{DSS}:
\begin{definition}\label{def:top:lin:ind}
A subset $A$ of a topological vector space $X$ is {\em topologically linearly independent} provided that for every neighborhood $W$ of $e_X$ there exists a neighborhood $U$ of $e_X$ such that for every finite $F\subset A$ and every indexed set $\{r_a\colon a\in F\}$ of reals, the inclusion $\sum_{a\in F}r_aa\in U$ implies that $r_aa\in W$ for all $a\in F$. We call this neighborhood $U$ a {\em $W$-witness} of the topological linear independence of $A$.
\end{definition}
In our paper, we explore the two previously defined notions. In the realm of topological vector spaces, where we have both, we compare them.

It is a folklore and simple fact that a subset $A$ of a group is independent if and only if the Kalton map $\kal{A}$ is an isomorphic embedding. That is, $A$ is independent if and only if $\hull{A}$ is (canonically isomorphic to) the direct sum $\bigoplus_{a\in A}\hull{a}$. The topological counterpart of this fact was proved in \cite{Spe} for precompact\footnote{Recall that a topological group is {\em precompact} if it is (isomorphic to) a subgroup of a compact group.} groups:

\begin{theorem}\cite[Theorem 1.1]{Spe}\label{fact:precompact}
 A subset $A$ of a precompact group is topologically independent if and only if $\hull{A}$ is the Tychonoff direct sum $\bigoplus_{a\in A}\hull{a}$.    
\end{theorem}

However, in general, the topological independence of a set $A$ only implies that the Kalton map $\kal{A}$ is open and injective (\cite[Proposition 4.7(i)]{DSS}) while it does not need to be continuous, as demonstrated in \cite[Example 3.1]{Spe}. This example uses the simple fact that independent and topologically independent sets coincide in discrete groups. Thus, any infinite independent subset $A$ of a discrete group is topologically independent. Nevertheless, Tychonoff direct sums of infinitely many groups are never discrete and therefore the discrete group $\hull{A}$ is unable to be the Tychonoff direct sum $\bigoplus_{a\in A}\hull{a}$. Since discrete groups are locally compact, this example demonstrates that local compactness cannot replace precompactness  in Theorem \ref{fact:precompact}. 

In our manuscript, we substantially generalize Theorem \ref{fact:precompact} in Theorem \ref{thm:mainone} by showing that a precompact group from Theorem \ref{fact:precompact} can be replaced by an arbitrary subgroup of $(\C^\times)^I$, where $I$ is any nonempty set. This is a generalization as every precompact group is (isomorphic to) a subgroup of $\S^I$ for an appropriate set $I$. Our result has several consequences. 

In Corollary \ref{cor:comp:generated}, we point out that precompact groups in Theorem \ref{fact:precompact} can be replaced by locally compact compactly generated groups as all such groups are subgroups of $(\C^\times)^I$ for a suitable set $I$. This positively answers the question from \cite[Remark 3.2]{Spe}. 

A further consequence is Theorem \ref{thm:positive:problem:solution:for:weak:top}, which completely resolves the relation of topological independence and Tychonoff direct sums in the realm of the underlying groups of topological vector spaces. It states that given a topologically independent subset $A$ of a topological vector space, the topological group $\hull{A}$ is the Tychonoff direct sum $\bigoplus_{a\in A}\hull{a}$ if and only if the topology of $\rhull{A}$ is weak. 

Another consequence of Theorem \ref{thm:mainone} is given in Section \ref{sec:top:lin:indep}, where we investigate various types of topological independence in topological vector spaces.
Namely, in Theorem \ref{thm:top:ind:top:lin:ind}, we show, among other things, that in topological vector spaces equipped with a weak topology, topologically independent sets coincide with topologically linearly independent sets. This is because, for a topologically (linearly) independent set in the presence of a weak topology, the topological vector space $\rhull{A}$ is simply the Tychonoff direct sum $\R^{(A)}$. Thus, we clarify the notion of topological linear independence for weak topologies. The topological and topological linear independence generally do not coincide, as demonstrated in Example \ref{ex:top:ind:not:top:lin:ind}. 

In the following, we turn to the situation in normed spaces. In Theorem \ref{thm:top:lin:ind:iff:unif:minimal}, we show that topologically linearly independent sets coincide with uniformly minimal sets in normed spaces. Uniform minimality is essential in studying bounded biorthogonal systems and Markushevich bases (see \cite{H, Ver}). Thus, by a simple topologization of linear independence in Definition \ref{def:top:lin:ind}, we link Tychonoff direct sums of the form $\R^{(A)}$ in topological vector spaces with weak topologies and uniformly minimal sets $A$ in normed (Banach) spaces. They are of the same nature -- the set $A$ is topologically linearly independent.

As an application of topological linear independence, we show in Theorem \ref{thm:free:basis:top:lin:ind:l1} that if a free basis $A$ of a Banach space $X$ recently introduced in \cite{PS2} is topologically linearly independent when we remove its base point, then $X$ is, in fact, isomorphic to the classical Banach space $\ell^1(A)$, and $A\setminus\{e_X\}$ is its (long) Schauder basis.

\section{Topological independence and Tychonoff direct sums}
The main result of this section is Theorem \ref{thm:mainone}. It enlarges the known class of topological groups, in which every topologically independent set $A$ generates the Tychonoff direct sum $\bigoplus_{a\in A}\hull{a}$, from the class of precompact groups (Theorem \ref{fact:precompact}) to the class of arbitrary subgroups of the topological group $(\C^\times)^I$, where $I$ is an arbitrary nonempty set and $(\C^\times)^I$ is equipped with the Tychonoff product topology. Its proof is based on the following result.
\begin{theorem}\cite[Theorem 5.1]{DSS}\label{fact:from:DSS}
A subset $A$ of a topological group $G$ is simultaneously topologically independent and absolutely Cauchy summable if and only if 
$\hull{A}$ is the Tychonoff direct sum $\bigoplus_{a\in A}\hull{a}$. 
\end{theorem}

Recall that, as in \cite{DSS}, a subset $A$ of a topological group $G$ is  {\em absolutely Cauchy summable\/} provided that for every neighborhood $V$ of the neutral element, there exists a finite set $F\subseteq A$ such that    
\begin{equation}\label{eq5}
\langle A\setminus F\rangle\subset V.
\end{equation}
Thus, Theorem \ref{fact:from:DSS} allows us to reduce the proof of Theorem \ref{thm:mainone} to show that every topologically independent subset $A$ of a subgroup of $(\C^\times)^I$ is absolutely Cauchy summable. This, in turn, is equivalent to the fact that the set $\{a\in A\colon a(i)\neq 1\}$ is finite for every $i\in I$ (see the proof of Theorem \ref{thm:mainone}), and it is proved in Lemma \ref{lemma:main}.

We will use the multiplicative notation to denote operations of the topological group $(\C^\times)^I$. In order to describe its topology, let us set the following notation. Given $M\subseteq I$  and $\varepsilon\in \R$, put $$\uar{M}{\varepsilon}\colon=\left\{f\in(\C^\times)^I\colon\left|\frac{f(i)}{|f(i)|}-1\right|<\varepsilon \mbox{ for all } i\in M\right\},$$
to denote the subset of $(\C^\times)^I$, which consists of elements with arguments at each coordinate $i\in M$ less than $\varepsilon$ far from $1$. Similarly, put $$\uab{M}{\varepsilon}\colon=\left\{f\in(\C^\times)^I\colon\left|{|f(i)|}-1\right|<\varepsilon \mbox{ for all } i\in M\right\},$$
to denote the subset of $(\C^\times)^I$, which consists of elements with absolute value at each coordinate $i\in M$ less than $\varepsilon$ far from $1$. Finally, define $$\mathcal{U}(M,\varepsilon)=\uar{M}{\varepsilon}\cap\uab{M}{\varepsilon}. $$

In this notation, the set $$\{\mathcal{U}(M,\varepsilon)\colon M\subset I \mbox{ is finite}, \varepsilon>0\}$$
forms a local base at the neutral element of the topology of $(\C^\times)^I$.

Lemma \ref{lemma:bourbaki} is a reformulation of the classical fact that the subgroups of the topological group $\R^n$ that have a rank greater than $n$ are nondiscrete.  
\begin{lemma}\label{lemma:bourbaki}
    Let $M$ be a finite nonempty set and $F$ a finite set of natural numbers. For every $j\in F$, fix $a_j\in\R^M$ arbitrarily. If the cardinality of $F$ is strictly greater than the cardinality of $M$, then for every neighborhood $U$ of zero in $\R^M$, there is $n\in F$, and for each $j\in F$, an integer $z_j$  such that
    \begin{equation}\label{eq:small:integer:combination}
    \sum_{j\in F}z_ja_j\in U, \mbox{ and } z_n\neq0.
    \end{equation}
\end{lemma}
\begin{proof}
   Fix a neighborhood $U$ of zero in $\R^M$.  There are two cases. 
   
   In the first case, there are integers $z_j$ for $j\in F$, not all equal zero, such that $$\sum_{j\in F}z_ja_j=0_{\R^M}.$$ This means that $a_j=a_k$ for some distinct $j,k\in F$ or the set $\{a_j\colon j\in F\}$ is not independent. In this case, \eqref{eq:small:integer:combination} holds for some $n\in F$. 

   In the other case, the set $S\colon=\{a_j\colon j\in F\}$ is independent, and $|F|=|S|$. Therefore, $|S|>|M|$. Consequently, it follows from \cite[Chap. VII, Theorem 1]{Bourbaki} that the topological group $\langle S\rangle$ is nondiscrete. Hence, there is, for each $j\in F$, an integer $z_j$ such that $$\sum_{j\in F}z_ja_j\in U\setminus\{0_{\R^M}\}.$$ This gives us $n\in F$ such that \eqref{eq:small:integer:combination} holds.
\end{proof}
\begin{lemma}\label{lemma:new:universal}
    Let $I$ be a nonempty set, $(c_n)_{n\in\N}$ a sequence in the topological group $(\C^\times)^I$, $M$ a finite subset of $I$, and $\varepsilon$ a positive real number. Under these conditions,  there are distinct $k,l\in\N$ such that 
\begin{equation}\label{eq:arg:univ}
    c_kc_l^{-1}\in\uar{M}{\varepsilon}.
\end{equation}
Further, for every finite $F\subset\N$ with cardinality greater than $|M|$, there is $n\in F$ and  for every $j\in F$ an integer $z_j$ such that 
    \begin{equation}\label{eq:1:universal}
    \prod_{j\in F}c_j^{z_j}\in \uab{M}{\varepsilon}, \mbox{ and } z_n\neq 0.
\end{equation}
\end{lemma}
\begin{proof}
    To find distinct $k,l\in\N$ such that \eqref{eq:arg:univ} holds, let $\pi\colon (\C^\times)^I\to \S^M$ be the projection defined as $$\pi(c)=\left(\frac{c(i)}{|c(i)|}\right)_{i\in M}$$ for all $c\in(\C^\times)^I$. If $\pi(c_k)=\pi(c_l)$ for some distinct $k,l\in\N$, then $k$ and $l$ satisfy \eqref{eq:arg:univ}. Otherwise, the set $\{\pi(c_k)\colon k\in\N\}$ is an infinite subset of the compact group $\S^M$. Since $\pi(\uar{M}{\varepsilon})$ is a neighborhood of the neutral element of $\S^M$, we can apply \cite[Lemma 2.3]{Spe}, whose proof is simple and straightforward, to find distinct $a,b\in\{\pi(c_k)\colon k\in\N\}$ such that $ab^{-1}\in\pi(\uar{M}{\varepsilon})$. Pick $k,l\in\N$ such that $\pi(c_k)=a$ and $\pi(c_l)=b$ and observe that \eqref{eq:arg:univ} is satisfied.

    Furthermore, arbitrarily fix a finite $F\subset\N$ of cardinality greater than $|M|$. To find $n\in F$ and integers $z_j$ for $j\in F$ such that \eqref{eq:1:universal} holds, let $\sigma\colon (\C^\times)^I\to \R_+^M$ denote the projection defined as $$\sigma(c)=(|c(i)|)_{i\in M}$$ for all $c\in (\C^\times)^I$. Since the multiplicative topological group $\R_+$ and the additive topological group $\R$ are isomorphic, we can fix an isomorphism $\varphi\colon\R_+^M\to\R^M$. Since $\sigma(\uab{M}{\varepsilon})$ is a neighborhood of identity of $\R_+^M$, it follows that $$U\colon=\varphi(\sigma(\uab{M}{\varepsilon}))$$ is a neighborhood of zero in $\R^M$. For each $j\in F$, put $a_j=\varphi(\sigma(c_j))$, and use Lemma \ref{lemma:bourbaki} to choose $n\in F$ and an integer $z_j$ for each $j\in F$, such that \eqref{eq:small:integer:combination} holds. Observe that \eqref{eq:1:universal} is satisfied.
\end{proof}

Before stating and proving the two key lemmas of this section, let us recall the definition of topological independence in the multiplicative notation and its immediate consequence: a subset $A$ of a topological group $G$ is {\em topologically independent}, provided that for every neighborhood $W$ of the neutral element $e$, there exists neighborhood $U$ of $e$ such that for every finite subset $F\subset A$ and every indexed set $\{z_a\colon a\in F\}$ of integers the inclusion $\prod_{a\in F}a^{z_a}\in U$ implies that $a^{z_a}\in W$ for all $a\in F$.
From the definition of topological independence, we obtain the following lemma.
 \begin{lemma}\label{lemma:A:top:ind:B:too}
     Let $A=\{a_n\colon n\in\N\}$ be a faithfully indexed topologically independent subset of a topological group $G$ and $(z_n)_{n\in\N}$ a sequence of nonzero integers. The set $B\colon =\{b_n\colon  b_n=a_n^{z_n}, n\in\N\}$ is a faithfully indexed topologically independent subset of $G$.
 \end{lemma}

Now we are ready to state and prove the two key lemmas of this section.

\begin{lemma}\label{lemma:previous}
   Let $I$ be a set of indices. If $A$ is a topologically independent subset of $(\C^\times)^I$, then the set $\{a\in A\colon |a(i)|\neq 1\}$ is finite for every $i\in I$.
\end{lemma}
\begin{proof}
    We will prove the contrapositive. Assume that there is a faithfully indexed subset $\{a_n\colon n\in\N\}$ of $A$ and $i\in I$ such that
\begin{equation}\label{eq:pibn:nonzero}
   |a_n(i)|\neq  1 \mbox{ for every } n\in\N.
\end{equation}
Let us show that $A$ is not topologically independent.

By \eqref{eq:pibn:nonzero}, for every $n\in\N$,  there is a nonzero integer $z_n$ such that
\begin{equation}\label{eq:mimo:W}
  a_n^z\not\in \u{\{i\}}{{1}/{2}} \mbox{ for each } z\in\Z \mbox{ such that }|z|\ge |z_n|.
\end{equation}
Put $B=\{b_n\colon b_n=a_n^{z_n},n\in\N\}$. If $b_k=b_l$ for some distinct natural numbers $k,l$, then $A$ is not topologically independent. Otherwise, $B$ is a faithfully indexed set, and, by Lemma \ref{lemma:A:top:ind:B:too}, it remains to show that $B$ is not topologically independent. Observe that from \eqref{eq:mimo:W} we get
\begin{equation}\label{eq:b:to:z:not:in:U}
b_n^z\not\in \u{\{i\}}{{1}/{2}} \mbox{ for each } z\in\Z\setminus\{0\} \mbox{ and } n\in\N.    
\end{equation}
 We claim that there is no $\u{\{i\}}{{1}/{2}}$-witness of the topological independence of $B$. 
 
 To show this, pick finite $M\subset I$ and positive $\varepsilon\in\R$ arbitrarily. It suffices to show that the basic neighborhood $\u{M}{\varepsilon}$ of the neutral element is not a 
$\u{\{i\}}{{1}/{2}}$-witness of the topological independence of $B$.

Let $\{F_n\colon n\in\N\}$ be a faithfully indexed partition of $\N$ such that each $F_n$ is finite and has a cardinality greater than $|M|$.

Chose positive $\delta\in\R$ such that 

\begin{equation}\label{eq:delta:for:epsilon}
    \uab{M}{\delta}\cdot\uab{M}{\delta}\subset\uab{M}{\varepsilon},
\end{equation}

and use the second part of Lemma \ref{lemma:new:universal} to find for every $k\in\N$ some $n_k\in F_k$ and, for  every $j\in F_k$, integers $z_j$ such that

\begin{equation}\label{eq:c:n}
c_k\colon =\prod_{j\in F_k}b_j^{z_j}\in\uab{M}{\delta}, \mbox{ and } z_{n_k}\neq 0.    
\end{equation}

From the first part of Lemma \ref{lemma:new:universal}, we get distinct $k,l\in\N$ such that \eqref{eq:arg:univ} holds.
Therefore, by \eqref{eq:c:n} and \eqref{eq:delta:for:epsilon} we get

$$c_kc_l^{-1}=\prod_{j\in F_k\cup F_l}b_j^{z_j}\in\left(\uab{M}{\delta}\cdot\uab{M}{\delta}\right)\cap\uar{M}{\varepsilon}\subset\u{M}{\varepsilon}.$$
 As $z_{n_k}\neq 0$ by \eqref{eq:c:n}, it follows from \eqref{eq:b:to:z:not:in:U} that 
 $b_{n_k}^{z_{n_k}}\not\in\u{\{i\}}{{1}/{2}}.$ 
  This shows that $\u{M}{\varepsilon}$ is not a 
  $\u{\{i\}}{{1}/{2}}$-witness 
  of the topological independence of $B$. Thus, $B$, and consequently $A$, is not topologically independent.
\end{proof}

\begin{lemma}\label{lemma:main}
Let $I$ be a set of indices. If $A$ is a topologically independent subset of $(\C^\times)^I$, then the set $\{a\in A\colon a(i)\neq 1\}$ is finite for every $i\in I$.
\end{lemma}
\begin{proof}
We will prove the contrapositive.  Assume that there is a faithfully indexed subset $\{a_n\colon n\in\N\}$ of $A$ and $i\in I$ such that\begin{equation}\label{eq:pibn:nonone}
   a_n(i)\neq 1 \mbox{ for every } n\in\N.
\end{equation}
Let us show that $A$ is not topologically independent.

By Lemma \ref{lemma:previous}, it suffices to assume that 
\begin{equation}\label{eq:abs:value:1}
  \mbox{the set }  \{n\colon |a_n(j)|\neq 1\} \mbox{ is finite for every } j\in I.
\end{equation}

It remains to show that there is no $\u{\{i\}}{1}$-witness of the topological independence of $A$. 
To show this, pick finite $M\subset I$ and positive $\varepsilon\in\R$ arbitrarily. We will prove that the basic neighborhood $\u{M}{\varepsilon}$ of the neutral element is not a $\u{\{i\}}{1}$-witness of the topological independence of $A$. 

Use \eqref{eq:abs:value:1} to find infinite 
 $L\subseteq\N$ such that 
\begin{equation} \label{eq:0real:part}
|a_n(j)|=1 \mbox{ for all } n\in L \mbox{ and } j\in M.    
\end{equation}

By \eqref{eq:pibn:nonone}, for every $n\in L$, there is an integer $z_n$ such that 
\begin{equation}\label{eq:znan:not:in:W}
    {a_n}^{z_n}\not\in\u{\{i\}}{1} \mbox{ for all } n\in L.
\end{equation}

For every $n\in L$ put $c_n={a_n}^{z_n}$. By Lemma \ref{lemma:new:universal}, there are distinct $k,l\in L$, such that \eqref{eq:arg:univ} holds. This together with \eqref{eq:0real:part} implies that $$c_kc_l^{-1}=a_k^{z_k}a_l^{-z_l}\in\uar{M}{\varepsilon}\cap\uab{M}{\varepsilon}=\u{M}{\varepsilon}.$$ 
Since $a_k^{z_k}\not\in\u{\{i\}}{1}$ by \eqref{eq:znan:not:in:W}, it follows that $\u{M}{\varepsilon}$ is not a $\u{\{i\}}{1}$-witness of the topological independence of $A$. This finishes the proof.
\end{proof}

\begin{theorem}\label{thm:mainone}
    Let $I$ be an arbitrary set. A subset $A$ of $(\C^\times)^I$ is topologically independent if and only if $\hull{A}$ is the Tychonoff direct sum $\bigoplus_{a\in A}\hull{a}$.  
\end{theorem}
\begin{proof}
    By Fact \ref{fact:from:DSS}, it suffices to prove that if $A$ is topologically independent, then it is absolutely Cauchy summable. To do so, pick a neighborhood $V$ of the neutral element of $(\C^\times)^I$. There are finite $M\subset I$ and positive $\varepsilon\in\R$ such that $$\u{M}{\varepsilon}\subset V.$$ By Lemma \ref{lemma:main}, the set $F=\{a\in A\colon a(i)\neq 1 \mbox{ for some } i\in M\}$ is finite. Observe that $\hull{A\setminus F}\subset \u{M}{\varepsilon}$. Therefore, \eqref{eq5} holds and $A$ is absolutely Cauchy summable. This finishes the proof. 
\end{proof}

\begin{corollary}\label{cor:comp:generated}
Let $G$ be a locally compact, compactly generated topological group. A subset $A$ of $G$ is topologically independent if and only if  $\hull{A}$ is the Tychonoff direct sum $\bigoplus_{a\in A}\hull{a}$.    
\end{corollary}
\begin{proof}
    By \cite[Theorem 9.8]{HR}, the topological group $G$ is isomorphic to $\R^k\times\Z^l\times F$ for appropriate nonnegative integers $k,l$ and compact group $F$. Hence $G$ is isomorphic to a subgroup of the topological group $\R^z\times F$ for some nonnegative integer $z$. In addition, the additive topological group $\R^z$ is isomorphic to the multiplicative topological group $\R_+^z$ and, by \cite[Corollary 11.5.2]{ADG}, the compact group $F$ is isomorphic to a topological subgroup of $\S^J$ for a suitable set $J$. Consequently, there is a set $I$ such that $G$ is isomorphic to a topological subgroup of $\R_+^I\times\S^I$. Since the latter topological group is isomorphic to $(\C^\times)^I$, the conclusion follows from Theorem \ref{thm:mainone}. 
\end{proof}

We end this section with Theorem \ref{thm:positive:problem:solution:for:weak:top}, which completely resolves the relation of topologically independent sets and Tychonoff direct sums in the realm of the underlying topological groups of topological vector spaces. 
Before stating it, recall that a topology of a topological vector space $X$ is {\em weak} provided that it is generated by the family $X^*$ of all continuous linear functionals on $X$.
\begin{theorem}\label{thm:positive:problem:solution:for:weak:top}
    Let $A$ be a topologically independent subset of a topological vector space $E$. Then $\hull{A}$ is the Tychonoff direct sum $\bigoplus_{a\in A}\hull{a}$ if and only if the topology of $\rhull{A}$ is weak.
\end{theorem}
\begin{proof}
    Put $X=\rhull{A}$. If the topological vector space $X$ has a weak topology, then the diagonal $$\Delta_{f\in X^*}f\colon X\to\R^{X^*}$$ is an isomorphic embedding of topological vector spaces. Since the topological groups $\R$ and $\R_+$ are isomorphic, $A$ can be considered to be a subset of the topological group $\R_+^{X^*}$ which in turn is (isomorphic to) a topological subgroup of $(C^\times)^{X^*}$. Hence, $\hull{A}$ is the Tychonoff direct sum $\bigoplus_{a\in A}\hull{a}$ by Theorem \ref{thm:mainone}. 

    If $\hull{A}$ is the Tychonoff direct sum $\bigoplus_{a\in A}\hull{a}$, then it is isomorphic to the Tychonoff direct sum $\Z^{(A)}$, because each $\hull{a}$ is isomorphic to the discrete group $\Z$. The latter Tychonoff direct sum has an invariant linear span by \cite[Corollary 3.5]{PS}, which means that its linear span is unique (up to an isomorphism of topological vector spaces), whenever it embeds as a topological group into a topological vector space. Therefore, the topological vector space $\rhull{A}$ is isomorphic to $\R^{(A)}$. Hence, it has a weak topology.    
\end{proof}

\begin{example}
    Every Schauder basis $A$ of an infinite-dimensional Banach space is topologically independent, but $\hull{A}$ is not the Tychonoff direct sum $\bigoplus_{a\in A}\hull{a}$.
\end{example}
\begin{proof}
    By \cite[Proposition 4.9]{DSS}, every Schauder basis in a Banach space is topologically independent. On the other hand, the topology of an infinite-dimensional normed space is never weak. The conclusion follows from Theorem \ref{thm:positive:problem:solution:for:weak:top}.
\end{proof}

\section{Topological (linear) independence and minimality in topological vector spaces} \label{sec:top:lin:indep}

We start this section by recalling a minimal set in a topological vector space.

\begin{definition}
 A subset $A$ of topological vector space $V$ is 
 {\em minimal} if, for all $a\in A$, we have
 \begin{equation}\label{eq:sembasic}
 a\not\in\overline{\langle A\setminus\{a\}\rangle_\R}.
 \end{equation}

\end{definition}

\begin{remark}
We have adopted the name {\em minimal} from \cite[Definition 6.1]{Singer}, where a minimal sequence was introduced in a Banach space. 
Minimal sets in Banach spaces are defined in \cite[Definition 1.1]{H}, where the connection of minimal sets and biorthogonal systems is explained via the Hahn-Banach theorem.
In \cite{Haz}, a minimal set is called {\em topologically free}. 
In the paper \cite{PS}, which we use substantially in this section, minimal sets are called {\em semi-basic}. This name comes from \cite{kal2}, where the minimal sequences have this name.
\end{remark}

In this section, we establish the following chains of implications.

\begin{center} \medskip\hspace{1em}\xymatrix{
\mbox{top. lin.  indep.}\ar@{=>}[r]^{\hspace{-.3cm}(a)}\ar@{=>}[d]^{(b)}&\mbox{top. independent}\ar@{=>}[d]^{(c)}\\
\mbox{minimal}  \ar@{=>}[r]^{\hspace{-.7cm}(d)}&\mbox{lin. independent} }
  \end{center}

In general, no other implication can be added to the above diagram as demonstrated in Examples \ref{ex:top:ind:not:top:lin:ind} and \ref{ex:minimal:not:top:ind} due to Eva Perneck\'{a}. On the other hand, implication (a) is reversible in the realm of topological vector spaces with weak topology (see Theorem \ref{thm:top:ind:top:lin:ind}). Implications (a) and (c) are established in Lemma \ref{lemma:finite:all:the:same}. The proof of implication (b) is a straightforward exercise. However, we establish the implication (b) in Corollary \ref{prop:lin:indep:is:semibasic}, which is based on Proposition \ref{prop:lemma:3.2} and Proposition \ref{prop:top:lin:ind:=equicontinuous}. These propositions provide reformulations of minimality and topological linear independence and provide a better insight into their close relationship. The implication (d) is obvious.

\begin{lemma}\label{lemma:finite:all:the:same}
Let $A$ be a subset of a Hausdorff topological vector space. Consider the following statements:
\begin{itemize}
    \item[(i)]
    $A$ is topologically linearly independent.
    \item[(ii)]
    $A$ is topologically independent.
     \item[(iii)]
    $A$ is linearly independent.
\end{itemize}
Then the following chain of implications holds: (i)$\Rightarrow$(ii)$\Rightarrow$(iii). Moreover, if $A$ is finite, then all three conditions are equivalent. 
\end{lemma}
\begin{proof}
The implication (i)$\Rightarrow$(ii) is obvious. The implication (ii)$\Rightarrow$(iii) is \cite[Proposition 4.11 (i)]{DSS}. Finally, if $A$ is finite, the implication (iii)$\Rightarrow$(i) is obvious.
\end{proof}

Given an element $a$ of a linearly independent subset $A$ of a vector space, the symbol $\pi^A_a$ denotes the unique projection from $\rhull{A}$ onto $\rhull{a}$ such that $\ker\pi^A_a=\rhull{A\setminus\{a\}}$. 

The proof of the following proposition can be found in the proof of \cite[Lemma 3.2]{PS}.
\begin{proposition}\label{prop:lemma:3.2}
    Let $A$ be a linearly independent subset of a topological vector space $E$. The following conditions are equivalent:
\begin{itemize}
    \item [(i)]
    $A$ is minimal.
    \item[(ii)]
    The projection $\pi_a\colon \rhull{A}\to\rhull{a}$ is continuous for each $a\in A$.
\end{itemize}
\end{proposition}

\begin{proposition}\label{prop:top:lin:ind:=equicontinuous}
Let $A$ be a linearly independent subset of a topological vector space $E$. The following conditions are equivalent:
\begin{itemize}
    \item [(i)]
    $A$ is topologically linearly independent.
    \item[(ii)]
    The collection $\{\pi_a^A\colon a\in A\}$ is equicontinuous.
\end{itemize}
\end{proposition}
\begin{proof}
It suffices to realize that given $x\in\rhull{A}$, there is a unique indexed set $\{r_a\colon a\in A\}$ of reals, where $r_a\neq 0$ only for finitely many $a\in A$, such that $$x=\sum_{a\in A}r_aa.$$  
Clearly, $$\pi^A_a(x)=r_aa.$$

Now, given a neighborhood $W$ of zero, the $W$-witness $U$ of the topological linear independence of $A$ is witnessing the equicontinuity of the collection $\{\pi^A_a\colon a\in A\}$ and vice versa.
\end{proof}

The next statement is a direct consequence of Propositions \ref{prop:lemma:3.2} and \ref{prop:top:lin:ind:=equicontinuous}.
\begin{corollary}\label{prop:lin:indep:is:semibasic}
Every topologically linearly independent set is minimal.
\end{corollary}

Example \ref{example:of:top:lin:ind} provides a typical topologically linearly independent set. As shown in Theorem \ref{thm:top:ind:top:lin:ind}, this is basically the only example of a topologically linearly independent set in a topological vector space with a weak topology.

\begin{example}\label{example:of:top:lin:ind}
Let $A$ be a set. For each $a\in A$, let $x_a\in\R^{(A)}$ satisfy $x_a(a)\neq 0$ and $x_a(b)=0$ for all $b\in A\setminus\{a\}$. The set $X=\{x_a\colon a\in A\}$ is topologically linearly independent, and $\rhull{X}$ is canonically isomorphic to the topological vector space $\R^{(A)}$.
\end{example}

\begin{theorem}\label{thm:top:ind:top:lin:ind}
Let $A$ be a subset of nonzero elements of a vector space $V$ equipped with a weak topology. The following statements are equivalent:
\begin{itemize}
    \item[(i)]
    $A$ is topologically linearly independent;
    \item[(ii)]
    $A$ is topologically independent;
    \item[(iii)]
    $A$ is absolutely Cauchy summable and minimal;
    \item[(iv)]
    The topologically group $\hull{A}$ is canonically isomorphic to the Tychonoff direct sum $\Z^{(A)}$;
    \item[(v)]
    The topological vector space $\rhull{A}$ is canonically isomorphic to the Tychonoff direct sum $\R^{(A)}$.
\end{itemize}
\end{theorem}
\begin{proof}
Item (i) implies (ii) by Lemma \ref{lemma:finite:all:the:same}. Assume (ii). 
By Theorem \ref{thm:positive:problem:solution:for:weak:top}, we have (iv).
The implication (iv)$\Rightarrow$(v) follows from \cite[Corollary 3.5]{PS}, and (v) implies (i) as shown in Example \ref{example:of:top:lin:ind}. 
Finally, the equivalence of (iii) and (v) is established in
\cite[Proposition 3.3]{PS}.  
\end{proof}

\begin{example}\label{ex:top:ind:not:top:lin:ind}
Consider the Banach space $\ell^{\infty}$ with $\|\cdot\|_{\infty}$-norm topology and its subset $A=\{e_1\}\cup\{e_1+e_n\colon \,n\in\N,\,n\geq 2\}$, where $e_i=(0,\ldots,0,\underbrace{1}_{i},0,\ldots)$ for every $i\in\N$. The set $A$ is topologically independent, but not minimal. In particular, according to Corollary \ref{prop:lin:indep:is:semibasic}, it is not topologically linearly independent.

Indeed, let $U$ be the open ball around $0_{\ell^{\infty}}$ with a radius of $1$. If $z_1e_1+\sum_{n=2}^kz_n(e_1+e_n)\in U$ for some $z_1,\ldots,z_k\in \Z$, then $$\max\left\{\left|\sum_{n=1}^kz_n\right|,|z_2|,\ldots,|z_k|\right\}<1.$$ Hence, $z_i=0$ for all $1\leq i\leq k$. Therefore, $U$ is a $W$-witness of the topological independence of $A$ for every neighborhood $W$ of $0_{\ell^{\infty}}$. 

On the other hand, if for any $\varepsilon>0$ we take $k\in\N$ such that $1/k<\varepsilon$, then $$e_1-\sum_{n=2}^{k+1}\frac 1k(e_1+e_n)=(0,\underbrace{-\frac 1k,\ldots,-\frac 1k}_{k\times},0\ldots)\in U(0_{\ell^{\infty}},\varepsilon).$$ Hence, $e_1\in\overline{\langle\{e_1+e_n\colon \,n\in\N,\,n\geq 2\}\rangle_{\R}}$ and $A$ is not minimal.
\end{example}

Our next example shows that the condition of absolute Cauchy summability cannot be omitted in item (iii) of Theorem \ref{thm:top:ind:top:lin:ind}.

\begin{example}\label{ex:minimal:not:top:ind}
Consider the Banach space $\ell^{\infty}$ with $w^\ast$-topology and its subset $A=\{e_1+e_n\colon \,n\in\N,\,n\geq 2\}$, where $e_i=(0,\ldots,0,\underbrace{1}_{i},0,\ldots)$ for every $i\in\N$. The set $A$ is minimal but not topologically independent. Indeed, let $n\geq 2$ and let $U=e_n^{-1}(-1,1)$ be a $w^\ast-$neighborhood of $0_{\ell^{\infty}}$ induced by functional $e_n\in\ell^1$. Then 
$$(e_1+e_n+U)\cap \langle \{e_1+e_i\colon \,i\in\N,\,i\geq 2, i\neq n\}\rangle_{\R}=\emptyset.$$ Hence, $e_1+e_n\notin \overline{\langle A\setminus\{e_1+e_n\}\rangle_{\R}}$ and $A$ is minimal. 

On the other hand, take the neighborhood $W=e_1^{-1}(-1,1)$ of $0_{\ell^{\infty}}$ given by the functional $e_1\in\ell^1$. We show that there is no $W$-witness for the topological independence of $A$. Let $\varphi_1=(y^1_1,y^1_2,\ldots),\ldots,\varphi_k=(y^k_1,y^k_2,\ldots)\in\ell^1$ and let $U=\bigcap_{i=1}^k\varphi_i^{-1}(-1,1)$. There exists $N\in\N$ such that $\sum_{n=N}^{\infty}|y^i_n|<1$ for all $1\leq i\leq k$. For $x=(e_1+e_N)-(e_1+e_{N+1})$, we get $|\varphi_i(x)|=|y^i_N-y^i_{N+1}|<1$ for every $1\leq i\leq k$. Hence $x\in U$ but $e_1+e_N\notin W$, which means that $U$ is not a $W$-witness of the topological independence of $A$.
\end{example}

\section{Topologically linearly independent sets in normed spaces}

In the first part of this section, we link topologically linearly independent sets with bounded biorthogonal systems and uniformly minimal sets. In order to do so, let us recall some necessary terminology.

Given a nonempty set $I$, a normed space $X$ and its dual $X^*$, a family $\{(x_i,x^*_i)\colon i\in I\}$ of pairs from $X\times X^*$ is called a {\em biorthogonal system}, provided that $x_i^*(x_j)=\delta_{ij}$ for all $i,j\in I$, where $\delta_{ij}$ is Kronecker $\delta$. This biorthogonal system is called {\em bounded} if there is positive $K\in\R$ such that $\norm{x_i}\cdot\norm{x_i^*}\le K$ for all $i\in I$. 

\begin{remark}\label{rem:bounded=unif:minimal}
    It is a simple application of the Hahn-Banach theorem that given a faithfully indexed set $\{x_i\colon i\in I\}$ in $X$ it is possible to find $x^*_i\in X^*$ for each $i\in I$ such that $\{(x_i,x^*_i)\colon i\in I\}$ forms a biorthogonal system if and only if the set $\{x_i\colon i\in I\}$ is minimal. 

Similarly, as remarked in \cite[Remark 1.24]{H}, the biorthogonal system above is bounded if and only if the set $\{x_i\colon i\in I\}$ is uniformly minimal in the sense of Definition \ref{def:unif:min}.
\end{remark}

\begin{definition}\label{def:unif:min}
    A subset $A$ of a normed space is {\em uniformly minimal} provided that there exists a positive $K\in\R$ such that for every $a\in A$ we have 
 \begin{equation}\label{eq:unif:minimal}
 \mbox{dist}\left( \frac{a}{\norm{a}},\overline{\langle A\setminus\{a\}\rangle_\R}\right)\ge K.
 \end{equation}
    
\end{definition}

\begin{remark}
    Uniformly minimal sequences in Banach spaces are studied, for example, in \cite{Ver, H}. They also appear in \cite{Singer}, however, the name ``uniformly minimal'' is suggested for a stronger property there (see \cite[Remark 7.2]{Singer}).
\end{remark}

It turns out that a subset of a normed space is topologically linearly independent if and only if it is uniformly minimal (Theorem \ref{thm:top:lin:ind:iff:unif:minimal}). This result is based on Proposition \ref{lemma:top:ind:in:banach}  which provides a simple description of topologically linearly independent sets in normed spaces. Its straightforward proof, which is based on the absolute homogeneity of a norm, is omitted. 
\begin{proposition}\label{lemma:top:ind:in:banach}
A subset $A$ of a normed space is topologically linearly independent if and only if there is positive $K\in\R$ such that for every finite $F\subset A$ and indexed set of reals $\{r_a\colon a\in F\}$ 
$$ \mbox{if } \norm{\sum_{a\in F}r_aa}<K,  \mbox{ then } ||r_aa||<1 \mbox{ for all } a\in F.$$
\end{proposition}

\begin{theorem}\label{thm:top:lin:ind:iff:unif:minimal}
    A subset of a normed space is topologically linearly independent if and only if it is uniformly minimal.
\end{theorem}
\begin{proof}
Assume that $A$ is a topologically linearly independent subset of a normed space, and let $K$ be as in Proposition \ref{lemma:top:ind:in:banach}. Pick $a\in A$ arbitrarily.  Since $\norm{\frac{a}{\norm{a}}}=1$, it follows from Proposition \ref{lemma:top:ind:in:banach} that $$\norm{\frac{a}{\norm{a}}-\sum_{b\in F}r_bb}\ge K$$ for every finite $F\subset A\setminus\{a\}$ and all reals $r_b$, where $b\in F$.  This gives us \eqref{eq:unif:minimal}, and $A$ is uniformly minimal.

To show the reversed implication, assume that $A$ is uniformly minimal and take $K$ as in \eqref{eq:unif:minimal}. Pick arbitrarily a finite $F\subset A$, and for each $a\in F$ a nonzero real number $r_a$ such that 
\begin{equation}\label{eq:sum:less:K}
\norm{\sum_{a\in F}r_aa}<K.    
\end{equation}

Fix $b\in F$ arbitrarily. From \eqref{eq:unif:minimal} and  \eqref{eq:sum:less:K} we get

$$K\le \mbox{dist}\left( \frac{b}{\norm{b}},\overline{\langle A\setminus\{b\}\rangle_\R}\right)\le\norm{\frac{r_bb}{\norm{r_bb}}+\sum_{a\in F\setminus\{b\}}\frac{r_a}{\norm{r_bb}}a}<\frac{K}{\norm{r_bb}}.$$

Multiplying the above inequalities by $\frac{\norm{r_bb}}{K}$ gives us $\norm{r_bb}<1$. Hence, $A$ is topologically linearly independent by Proposition \ref{lemma:top:ind:in:banach}.
\end{proof}

\begin{remark}\label{rem:card:of:top:lin:ind}
    In this remark, we estimate cardinalities of topologically linearly independent sets in some topological vector spaces.
    
    Recall that a {\em Markushevich basis} of a Banach space $X$ is a biorthogonal system $\{(x_i,x^*_i)\colon i\in I\}$ such that $\rhull{\{x_i\colon i\in I\}}$ is dense in $X$ and $\rhull{\{x^*_i\colon i\in I\}}$ is dense in $X^*$ in the $w^*$-topology. By Remark \ref{rem:bounded=unif:minimal} and Theorem \ref{thm:top:lin:ind:iff:unif:minimal}, Markushevich basis is bounded if and only if the set $\{x_i\colon i\in I\}$ is topologically linearly independent. 
\begin{itemize}
        \item[(i)] Every infinite-dimensional Banach space contains a countably infinite topologically linearly independent set, as it contains a bounded Markushevich basis by a result of Pe\l czy\'{n}ski \cite[Theorem 1]{Pelc}. See also \cite{PL1}.
                
        \item [(ii)] Under the continuum hypothesis, there is a non-separable Banach space $X$ such that every minimal (in particular, every topologically linearly independent) set in $X$  is countable. This
is directly based on Kunen’s result, which can be found in \cite[Theorem 7.7]{Negrepontis}. 
        
        \item[(iii)] To find topologically linearly independent sets of uncountable cardinality in some Banach spaces, we can use the result of Plichko from \cite{PL}, which states that every non-separable Banach space with a Markushevich basis contains a bounded Markushevich basis (see also \cite[Theorem 5.13]{H}).

        \item[(iv)] As we will see in Proposition \ref{PUB}, there are infinite-dimensional topological vector spaces such that every topologically linearly independent set is finite.
    \end{itemize}
\end{remark}

As usual, given a normed space $X$, we denote by $w$ the topology of $X$ generated by all continuous functionals on $X$.

\begin{proposition}\label{PUB}
Let $X$ be a normed space. Every topologically linearly independent subset of $(X,w)$ is finite.
\end{proposition}
\begin{proof}
Assume, for contradiction, that there is an infinite faithfully indexed set $A\colon =\{a_n\colon n\in\N\}\subset (X,w)$ that is topologically linearly independent. For each $n\in\N$ choose $r_n\in\R$ such that 
\begin{equation}\label{eq;big:norm}
\norm{r_na_n}>n.    
\end{equation}
 This is possible because all $a_n$'s are nonzero, being elements of a topologically linearly independent set.   Since the topology $w$ is weak, it follows from Theorem \ref{thm:top:ind:top:lin:ind} that the set $A$ is absolutely Cauchy summable. Therefore, for each neighborhood $V$ of zero in $(B,w)$, there is finite $F\subset A$ such that \eqref{eq5} holds. It follows that the sequence $(r_na_n)_{n\in\N}$ converges weakly to zero. Thus, it is norm-bounded by the uniform boundedness principle. This contradicts \eqref{eq;big:norm}.  
\end{proof}

\begin{example}
Continuous linear injections do not necessarily preserve topological linear independence. 
\end{example}
\begin{proof}
Let $B$ be an infinite-dimensional Banach space with a bounded Markushevich basis $A$. Then $A$ is an infinite topologically linearly independent subset of $B$ by Remark \ref{rem:card:of:top:lin:ind}.
Consider the identity map $\iota\colon (B,|| \hspace{0.1cm} ||)\to (B,w)$. Then $\iota$ is a continuous linear injection, and $\iota(A)$ being infinite is not topologically linearly independent in $(B,w)$ by Proposition~\ref{PUB}.
\end{proof}

We close this paper with an observation that combining topological linear independence with a property of being a free basis ends up with the notion of a (long) Schauder basis of the classical Banach space $\ell^1(N)$ of all elements $x=(x_n)_{n\in N}$ such that $x_n$ is nonzero for at most countably many $n\in N$ and $\sum_{n\in N}|x_n|$ converges unconditionally to the norm of $x$.

Recall that, as in \cite{PS2},  a {\em free basis} $N$ of a Banach space $X$ is a closed subset of $X$ containing $0_X$ with the property that for every Banach space $Y$ and every Lipschitz map $f\colon N\to Y$ such that $f(0_X)=0_Y$ there exists unique bounded linear extension $\hat{f}\colon X\to Y$ of $f$.

\begin{theorem}\label{thm:free:basis:top:lin:ind:l1}
    If $M$ is a topologically linearly independent subset of a Banach space $X$ and $M\cup\{e_X\}$ is a free basis of $X$, then there is a canonical isomorphism $\phi\colon X\to \ell^1(N)$, where $N=\left\{\frac{m}{\norm{m}}\colon m\in M\right\}$. In particular,  $M$ is a (long)  Schauder basis of $X$.
\end{theorem}
\begin{proof}
   From \cite[Lemma 3.1]{PS2}, it follows that $N\cup\{e_X\}$ is a free basis of $X$. From the definition, the set $N$ is topologically linearly independent. Consequently, according to Theorem \ref{thm:top:lin:ind:iff:unif:minimal}, it is uniformly minimal. Hence, $N$ is uniformly discrete by \eqref{eq:unif:minimal}. Since $N$ is also bounded as all its elements have norm 1, it is bi-Lipschitz equivalent to a metric space in which all elements have a distance equal to 2. To conclude the proof, it remains to apply \cite[Example 3.10]{W2} together with a dictionary from \cite[Observation 2.2]{PS2}.
\end{proof}

\medskip
\begin{flushleft}
\textbf{Acknowledgment}
\end{flushleft}
\smallskip
The author would like to thank Eva Perneck\'{a} for providing Examples
\ref{ex:top:ind:not:top:lin:ind} and \ref{ex:minimal:not:top:ind}.

\end{document}